\documentclass[11pt]{amsart}

\usepackage{amsthm}
\usepackage{amssymb}
\usepackage{graphicx}									 
\usepackage{stmaryrd}                  
\usepackage{footmisc}                  
\usepackage{paralist}
\usepackage{wrapfig}
\usepackage[draft]{pdfcomment}
\usepackage{bbm}
\usepackage[arrow, matrix, curve]{xy}
\usepackage{color}
\usepackage{todonotes}

\usepackage{pict2e}

\usepackage{enumitem}

\newtheorem{thm}{Theorem}[section]

\newtheorem{lem}[thm]{Lemma}

\newtheorem{cor}[thm]{Corollary}
\newtheorem{prp}[thm]{Proposition}

\newtheorem{maintheorem}{Theorem}

\theoremstyle{definition}

\newtheorem{qst}[thm]{Question}
\newtheorem{rem}[thm]{Remark}

\theoremstyle{remark}

\newcommand{\Co}{\mathbb{C}}
\newcommand{\R}{\mathbb{R}}
\newcommand{\RR}{\mathbb{R}}

\newcommand{\Z}{\mathbb{Z}}
\newcommand{\ZZ}{\mathbb{Z}}

\newcommand{\orb}{{\mathrm{orb}}}

\newcommand{\To}{\rightarrow}

\newcommand{\Orb}{\mathcal{O}}
\newcommand{\Or}{\mathrm{O}}
\newcommand{\SOr}{\mathrm{SO}}
\newcommand{\SUr}{\mathrm{SU}}

\title[Highly connected orbifolds]{How highly connected can an orbifold be?}

\author[C. Lange]{Christian Lange}

\address{Christian Lange, Ludwig-Maximilians-Universit\"at M\"unchen, Mathematisches Institut\newline\indent Theresienstra{\ss}e 39, 80333 München, Germany}
\address{Ruhr-Universität Bochum, Mathematisches Institut\newline\indent Universitätsstr. 150, 44780 Bochum, Germany}
\email{lange@math.lmu.de, christian.lange@ruhr-uni-bochum.de}

\author[M. Radeschi]{Marco Radeschi}
\address{Marco Radeschi, University of Notre Dame,\newline\indent Department of Mathematics, 255 Hurley, Notre Dame, IN 46556 }
\email{mradesch@nd.edu}
\thanks{}

\begin{document}

\maketitle

\begin{abstract}
On the one hand, we provide the first examples of arbitrarily highly connected (compact) bad orbifolds. On the other hand, we show that $n$-connected $n$-orbifolds are manifolds. The latter improves the best previously known bound of Lytchak by roughly a factor of $2$. For compact orbifolds and in most dimensions we prove slightly better bounds. We obtain sharp results up to dimension $5$.
\end{abstract}


\section{Introduction}

Orbifolds are a generalization of manifolds that incorporate local symmetries. They were introduced by Satake under the name of V-manifolds \cite{MR0079769,Satake} and later rediscovered by Thurston. Orbifolds for instance occur as quotient spaces of Lie group actions or foliations \cite{LW16}, as collapsed limits of manifolds under Gromov-Hausdorff convergence \cite{MR1145256} and as moduli spaces, e.g. in Teichmüller theory. 

Thurston defined the notion of orbifold coverings and showed the existence of an orbifold version of coverings and fundamental groups analogous to the classical one \cite{Thurston}. More generally, every orbifold can be thought as the quotient of an almost free action of a compact Lie group on a manifold, and this allows to define algebraic invariants, like \emph{orbifold homotopy}- and \emph{(co)homology groups}, simply as the usual homotopy and (co)homology groups of the associated Borel construction. In the special case of the orbifold fundamental group one indeed recovers Thurston's definition, because orbifold coverings are in one-to-one correspondence with coverings of a model for the Borel construction.

An orbifold is called \emph{good} if it is covered by a manifold and \emph{bad} otherwise. It is a natural question to ask which measurable parameters tell good and bad orbifolds apart
or identify manifolds among orbifolds. Such identifiers can have also practical applications. The fact that an orbifold which admits a Riemannian metric of constant curvature is always good \cite{MM91} was for instance applied in \cite{Le15} in the context of Alexandrov geometry. In \cite{LW16} Lytchak and Wilking completed the program initiated by Ghys \cite{Gh84}  and Gromoll--Grove \cite{GG88} of classifying Riemannian foliations of spheres, and a large part of the work amounted to showing that a compact $7$-connected $8$-orbifold is in fact a manifold. In the context of Riemannian orbifolds all of whose geodesics are closed \cite{ALR21} the authors applied a cohomology characterization of manifolds among orbifolds, see Proposition \ref{prp:orbifold_manifold}, in order to show that odd-dimensional orbifolds all of whose geodesics are closed are covered by spheres.

In terms of topological parameters, Davis asked whether a contractible orbifold is a manifold \cite[p.~28]{Da11}. A positive answer to this question was provided by Lytchak, who actually showed that a $(2n-2)$-connected $n$-dimensional orbifold is a manifold \cite{Ly13}, see also Section \ref{sub:Lytchak_bound}. In the same paper, \cite{Ly13} Lytchak moreover asked whether arbitrarily highly connected bad orbifolds exist.

In this paper we answer Lytchak's question in the positive, and we improve his bound by roughly a factor of $2$.

On improving the bound we have the following result:

\begin{maintheorem} \label{thm:main_noncompact}
A $n$-connected $n$-orbifold is a manifold for any $n \geq 1$. A $2n$-connected $(2n+1)$-orbifold is a manifold for any $n \geq 1$.
\end{maintheorem}

Note that the even-dimensional part of this theorem is only a statement about non-compact orbifolds: A compact $n$-orbifold has a fundamental class and can thus not be $n$-connected. In the compact case, we obtain a better bound by additionally exploiting Lefschetz duality and the classification of homotopy spheres that are covered by spheres \cite{Sje73}.

\begin{maintheorem} \label{thm:main_compact} A compact $(2n-2)$-connected $2n$-orbifold is a manifold for any $n \geq 3$. For $n \geq 3$ not a power of $2$ a compact $(2n-2)$-connected $(2n+1)$-orbifold is a manifold.
\end{maintheorem}

In particular, our argument simplifies the tricky part (i.e. Section 4) in the work \cite{LW16} of Lytchak and Wilking concerned with showing that a certain compact $7$-connected $8$-orbifold is a manifold. Moreover, our proofs of Theorem \ref{thm:main_noncompact} and \ref{thm:main_compact} show that for orbifolds with nonempty strata only in sufficiently low codimensions their conclusions still hold when assuming less connectivity.

On the constructive side we answer Lytchak's question about the existence of highly connected bad orbifolds in the following way.

\begin{maintheorem} \label{thm:main_existence} For any $n\geq 4$, there exist compact $\lfloor n/2-1\rfloor$-connected bad orbifolds in dimension $n$.
\end{maintheorem}

Moreover, in low dimensions we present some specific constructions that yield higher connectedness than provided by Theorem \ref{thm:main_existence}. For instance, there exists a compact, $3$-connected, bad $4$-orbifold. All our results are summarized in the following section.

\subsection{Summary of the results}
Given a dimension $n$, let $\kappa(n)$ be the maximum $k$ such that there exists a bad, $k$-connected, $n$-orbifold, and let $\kappa_c(n)$ be the maximum $k$ such that there exists a bad, $k$-connected, compact $n$-orbifold. Then the statements we prove can be written as follows:

\begin{maintheorem} [Bounds for $\kappa(n)$] The following hold:
\begin{enumerate}
\item $\left\lfloor{n\over 2}\right\rfloor-1\leq\kappa(n)<2\left\lfloor{n\over 2}\right\rfloor$
\item $\kappa(3)=1$, $\kappa(4)=\kappa(5)=3$.
\end{enumerate}
\end{maintheorem}

\begin{maintheorem} [Bounds for $\kappa_c(n)$] The following hold:
\begin{enumerate}
\item $\left\lfloor {n\over 2}\right\rfloor-1\leq\kappa_c(n)<2\left\lfloor {n-1\over 2}\right\rfloor$ 
\item If $n\neq 2^k$, $\kappa_c(2n+1)< 2n-2$.
\item $\kappa_c(3)=1$ and $\kappa_c(4)=3$.
\end{enumerate}
\end{maintheorem}

\subsection{Structure of the paper} In Section \ref{sec:prelimiaries} we recall some background about orbifolds. In particular, we recall a cohomological characterization of manifolds among orbifolds, see Proposition \ref{prp:orbifold_manifold}.  In Section \ref{sec:construction_examples} we provide specific examples of highly connected orbifolds in low dimensions. Moreover, we prove Theorem \ref{thm:main_existence} about the existence of arbitrarily highly connected orbifolds. Finally, in Section \ref{sec:obstructions} we prove our Theorems \ref{thm:main_noncompact} and \ref{thm:main_compact}.
\\

Let us close the introduction with the following questions:

\begin{qst} What are the precise values of $\kappa(n)$ and $\kappa_c(n)$?
\end{qst}
\begin{qst}  (When) does $\kappa=\kappa_c$ hold?
\end{qst}

\textbf{Acknowledgements.} We are grateful to Alexander Lytchak for comments on a draft version of this paper.

\section{Preliminaries} \label{sec:prelimiaries}

\subsection{Orbifolds} An \emph{$n$-dimensional Riemannian orbifold} is a metric length space $\Orb$ such that each point in $\Orb$ has a neighborhood that is isometric to the quotient of an $n$-dimensional Riemannian manifold $M$ by an isometric action of a finite group $\Gamma$. Every such Riemannian orbifold has a canonical smooth orbifold structure in the usual sense \cite{La20}. Conversely, every smooth (effective) orbifold can be endowed with a Riemannian metric, and then the induced length metric turns it into a Riemannian orbifold in the above sense. For a point $p$ in $\Orb$ the isotropy group of a preimage of $p$ in a Riemannian manifold chart is uniquely determined up to conjugation in $\Or(n)$. Its conjugacy class  is called the \emph{local group} of $\Orb$ at $p$ and we also denote it as $\Gamma_p$. The point $p$ is called \emph{regular} if this group is trivial and \emph{singular} otherwise. More precisely, an orbifold admits a stratification into manifolds, where the stratum of codimension $k$ is given by
\[
			\Sigma_k = \{ p\in \Orb \mid \mathrm{codim} \mathrm{Fix}(\Gamma_p)=k \}.
\]
In particular, $\Sigma_0$ is the set of regular points.

Riemannian orbifolds for instance arise as a quotient of a Riemannian manifold $M$ by an effective, isometric and almost free (i.e. isotropy groups are finite) action of a compact Lie group $G$. In fact, every Riemannian orbifold arises in this way. Namely, it can be obtained as the quotient of its (orthonormal) frame bundle by the natural $\Or(n)$-action on it. The homotopy type of the corresponding Borel construction $B\Orb := M\times_G EG$, which was first considered by Haefliger \cite{Hae84}, depends only on the orbifold $\Orb$ and not on the specific representation of $\Orb$ as a quotient, see \cite[Proposition~1.51]{AdemLeidaRuan}. Hence, if $\Orb$ is a manifold, we can take $M=\Orb$, $G=\{1\}$ and see that $B\Orb \simeq M$. The orbifold homotopy groups, cohomology groups, etc. of an orbifold $\Orb$ are defined as the corresponding invariants of $B\Orb$. In particular, an orbifold $\Orb$ is by definition $k$-connected, if $B\Orb$ is so. 

The notion of an orbifold covering is defined in \cite{Thurston}. An alternative metric characterization of orbifold coverings is provided in \cite{La20}.

In dimension $1$ every orbifold is covered by $\R$. In dimension $2$ a so-called \emph{football} orbifold (also known as spindle or complex weighted projective line) provide examples of simply connected, bad, compact $2$-orbifold, see Figure \ref{fig:spindle}. In fact, any bad, simply connected $2$-orbifold is of this type \cite{Da11}. 

\begin{figure}
	\centering
		\def\svgwidth{0.11\textwidth}
		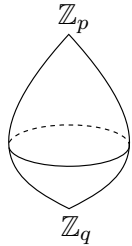
	\caption{A football orbifold is homeomorphic to $S^2$ and has two cyclic singularities of order $p$ and $q$. Its fundamental group is isomorphic to $\Z_{\mathrm{gcd}(p,q)}$.}
	\label{fig:spindle}
\end{figure}

An orbifold (resp. Riemannian orbifold) is diffeomorphic to a manifold (resp. isometric to a Riemannian manifold) if and only if its cohomology groups are non-trivial in only finitely many degrees. This statement appears in a more general form in the work of Quillen \cite[Corollary~7.8]{Qu71}. In particular, it yields an alternative proof for Davis' question whether a contractible orbifold is a manifold \cite[p.~28]{Da11}. In the special case of an orientable orbifold an easier argument is provided in \cite[Proposition~3.3]{ALR21}. For the convenience of the reader we present this argument here in a condensed form. 

\begin{prp}\label{prp:orbifold_manifold} An orientable $n$-orbifold $\Orb$ is a manifold, i.e. all its local groups are trivial, if and only if $H_{\orb}^i(\Orb)=H^i(B\Orb;\Z)=0$ for all $i>n$. This is the case if and only if the cohomology is nontrivial only in finitely many degrees.
\end{prp} 
\begin{proof} We only need to show the if part. If $\Orb$ is not a manifold, then there exists some $\Z_p \subset \SOr(n)$ which fixes some $x \in \mathrm{Fr}(\Orb)$. In this case the projection $\mathrm{Fr}(\Orb)_{\Z_p}:=\mathrm{Fr}(\Orb) \times_{\Z_p} E\SOr(n) \To B \Z_p$ admits the section $s(b)=[x,b]$. Hence we obtain an inclusion $H^*(B \Z_p) \hookrightarrow H^*(\mathrm{Fr}(\Orb)_{\Z_p})$. Now consider the fibration
\[
		\SOr(n)/\Z_p \To \mathrm{Fr}(\Orb)_{\Z_p} \To B\Orb:=\mathrm{Fr}(\Orb)_{\SOr(n)}
\]
Since $\Orb$ is orientable, $\pi_1(B\Orb)$ acts trivially on $H^*(\SOr(n)/\Z_p)$ \cite[Lemma~3.1]{ALR21} and thus the Leray--Serre spectral sequence satisfies
\[
		E_2^{p,q}=H^p(B\Orb;H^q(\SOr(n)/\Z_p))
\]
and converges to $H^{p+q}(\mathrm{Fr}(\Orb)_{\Z_p})$. Since $\SOr(n)/\Z_p$ is a finite-dimensional manifold and $H^*(\mathrm{Fr}(\Orb)_{\Z_p})$ is nontrivial in infinitely many degrees as we have seen above, the same must be true for $H^*(B\Orb)$.
\end{proof}
\begin{rem} An orbifold can be homeomorphic to a manifold even if some of its local groups are not trivial. The question when this happens is completely answered in \cite{La19}.
\end{rem}

\subsection{Lytchak's bound}\label{sub:Lytchak_bound} We recall the following argument by Lytchak \cite{Ly13}. Suppose an $n$-orbifold $\Orb$ is $k$-connected for some $k\geq 1$. Then $\Orb$ is in particular orientable and so one can realize it as an almost free $G=\SOr(n)$ quotient, of a manifold $M$. The orbit map $o_p: G \To M$ is then $k$-connected as well. Thus for any $l<k$ the map $H^l(M) \To H^l(G)$ is surjective. Since the orbit map factorizes through $\pi_p : G \To G/G_p$, the map $H^l(G/G_p;\Z) \To H^l(G;\Z)$ is surjective for $l<k$ as well. If $\Orb$ is not a manifold, i.e. if $G_p$ is nontrivial for some $p$, then $G \To G/ G_p$ is a nontrivial covering map for this $p$, in which case the image of $H^m(G/G_p;\Z)$ in $H^m(G;\Z) = \Z$ is a subgroup of index $|G_p|$ for $m=\dim G = n(n-1)/2$. Since the free part of $H^*(G;\Z)$ is generated in degree $2n-3$, it follows that $\Orb$ is a manifold if $k\geq 2n-2$.

We remark that this argument can be improved for $4$-orbifolds. Indeed, in this case $H^6(G;\Z)$ is generated in degree $3$, not just in degree $\leq 2n-3 =5$.
Therefore, the above argument shows that a $4$-orbifold $\Orb$ is a manifold, if it is $4$-connected. In Section \ref{sec:obstructions} we will see an alternative proof of this conclusion that works in any dimension.

\section{Examples \& Constructions} \label{sec:construction_examples}

In this section we provide specific examples and general constructions for highly connected orbifolds.

\subsection{Low dimensional examples}

An example of a compact, simply connected, bad $2$-orbifold is shown in Figure $\ref{fig:spindle}$ when $p$ and $q$ are coprime. We will in particular see that compact, simply connected, bad orbifolds exist in all dimensions. An example in dimension $3$ is provided by the following proposition.

\begin{prp} There exists a compact simply connected bad $3$-orbifold.
\end{prp}
\begin{proof}
There exists a $3$-orbifold with underlying space $S^3$ whose singular set is the trivalent graph shown in Figure \ref{fig:bad_3_orbifold} \cite{Thurston}. The edge weights specify the orders of the corresponding cyclic local groups. The local groups at the two upper vertices are the orientation preserving symmetry group $\mathrm{I}$ of an icosahedron. The local groups at the two lower vertices are the orientation preserving symmetry group $\mathrm{O}$ of a cube. We claim that this orbifold is simply connected as an orbifold. To see this we first observe that it is the double of the dashed $3$-ball. By Seifert--Van Kampen it suffices to show that this $3$-ball is simply connected as an orbifold. The latter is covered by two good orbifolds with fundamental group $\mathrm{I}$ and $\mathrm{O}$, respectively. Their intersection is homotopy equivalent to $\R^2(2,3)$, an open disk with two cyclic orbifold singularities of order $2$ and $3$, respectively. The fundamental group of this disk has two generators $x$ and $y$ represented respectively by a loop around the cyclic singularity of order $2$, and a loop around the cyclic singularity of order $3$. A presentation of this fundamental group is given by $\left\langle x,y \mid x^2,y^3\right\rangle$. The fundamental groups of the two good orbifold balls isomorphic to $\mathrm{O}$ and $\mathrm{I}$ are generated by $x$ and $y$ as well and have presentations $\left\langle x,y \mid x^2,y^3,(xy)^4\right\rangle$ and $\left\langle x,y \mid x^2,y^3,(xy)^5\right\rangle$, respectively. Applying Seifert--Van Kampen again now proves the claim.
\end{proof}

\begin{figure}
	\centering
		\def\svgwidth{0.4\textwidth}
		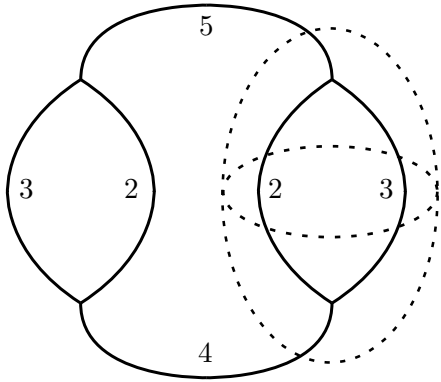
	\caption{Compact simply connected bad $3$-orbifold}
	\label{fig:bad_3_orbifold}
\end{figure}

In dimension $4$ higher connectedness can be achieved.

\begin{prp} \label{prp:3c_dimension4} There exists a compact $3$-connected bad $4$-orbifold.
\end{prp}
\begin{proof}
If $\rho$ is an irreducible representation of $\SUr(2)$ on a complex even-dimensional vector space $\Co^n$, then the induced action of $\SUr(2)$ on the unit sphere in $\Co^n$ is almost free but not free, and so the corresponding quotient, which is also referred to as a \emph{weighted quaternionic projective space}, is an orbifold \cite[p.~154]{GG88}. In particular, for $n=4$ the quotient is a compact $4$-orbifold that is not a manifold. Moreover, since $S^7$ is $3$-connected and $\SUr(2)\cong S^3$ is $2$-connected, the long exact sequence in homotopy implies trivially that this $4$-orbifold is $3$-connected.
\end{proof}

Note that by taking products with $\R^n$ Proposition \ref{prp:3c_dimension4} provides examples of (noncompact) $3$-connected $n$-orbifolds for any $n\geq 5$. In dimension $5$ we are only aware of $2$-connected compact examples. In fact, the example in the following proposition can be shown to be not $3$-connected.

\begin{prp} There exists a compact $2$-connected bad $5$-orbifold.
\end{prp}
\begin{proof}
Consider the maps $\rho_1: \SUr(2)\to \SUr(3)$ given by inclusion, and the map $\rho_2:\SUr(2)\to \SUr(3)$ given by the composition $\SUr(2)\to \SOr(3)\to \SUr(3)$. Finally, consider the action of $\SUr(2)$ on $\SUr(3)$ given by $g\cdot h= \rho_1(g)h(\rho_2(g))^{-1}$, which was proved to be almost free but not free by Yeroshkin \cite{Ye15}. In particular the biquotient $\mathcal{\Orb}:=\SUr(3)/\!/\SUr(2)$ is a 5-dimensional orbifold. Since both $\SUr(3)$ and $\SUr(2)$ are 2-connected, the long exact sequence in homotopy implies as above that $\mathcal{\Orb}$ is $2$-connected as well.
%
%
%
\end{proof}

\subsection{Arbitrarily highly connected bad orbifolds.}

Here we present a construction of arbitrarily highly connected orbifolds. It is based on the existence of certain parallelizable lens spaces as well as a surgery construction by Milnor that allows to kill homotopy groups.

\begin{prp}
For any $n\geq 4$, there exist compact $\lfloor n/2-1\rfloor$-connected bad orbifolds in dimension $n$.
\end{prp}
\begin{proof}
The proof is constructive. Let $M$ be either a stably parallelizable lens space $M:=S^{n-1}/\ZZ_p$, whose existence is guaranteed by \cite{EMSS:77}, or the parallelizable product $M:=S^{n-2}/\ZZ_p \times S^1$ of such a lens space with $S^1$, depending on whether $n$ is even or odd.

Take $W:=M\times [0,1]$, a compact parallelizable $n$-manifold with $\partial W=M \cup M$. By the corollary of \cite[Theorem~2]{Mil61} it is possible to perform surgery and turn it into a $\lfloor n/2-1\rfloor$-connected manifold $W'$, with $\partial W'=M \cup M$.

Finally, let $\mathcal{O}$ be the orbifold obtained by gluing two copies of $X:=\mathbb{D}^{n}/\ZZ_p$ to $W'$ along their boundary. We claim that $\mathcal{O}$ is $\lfloor n/2-1\rfloor$-connected as an orbifold. Since $W'$ is simply connected, it follows that $\pi_1^{\orb}(\Orb)=1$ by Seifert--Van Kampen. Moreover, since for $i=1,\ldots \lfloor n/2-1\rfloor$ we have that $H_i^\orb(W')=H_i(W')=0$ and that ${i_1}_*:H_i^\orb(M\cup M)\to  H_i^\orb(X \cup X)$ is an isomorphism

for $i< n-1$, cf. \cite[Proposition~3.5.]{La19}, Mayer-Vietoris implies that $H_i^{\orb}(\Orb)=0$ for $i=1,\ldots, \lfloor n/2-1\rfloor$. Now Hurewicz implies that $\Orb$ is $\lfloor n/2-1\rfloor$-connected.
\end{proof}

\begin{rem}
Killing homotopy groups via surgery becomes more delicate past degree $\left\lfloor{n/ 2}-1\right\rfloor$ since new surgery might reintroduce homology in lower degrees, and it involves analyzing the intersection form in $H_{\left\lfloor{n/ 2}\right\rfloor}(W)$ (when $\dim W$ is even) or the linking form in $\mathrm{Tor}(H_{\left\lfloor{n/ 2}\right\rfloor}(W))$ (when $\dim W$ is odd), see Milnor \cite{Mil61} and Wall \cite{Wal62}. In these cases, however, one also needs vanishing conditions on $\partial W$ such as $H_{\left\lfloor{n/ 2}\right\rfloor}(\partial W)=H_{\left\lfloor{n/ 2}\right\rfloor+1}(\partial W)$ which are in general not satisfied in the examples above.
\end{rem}

\section{Stratification \& Obstructions}\label{sec:obstructions}
The first step in proving our obstruction results is to rule out the presence of $2$-torsion in sufficiently highly connected orbifolds. The following lemma generalizes a conclusion in the proof of \cite[Lemma~3.2]{LW16}.

\begin{lem}\label{L:improved-no-2-torsion}
Let $\Orb$ be a $2^{a-1}$-connected $n$-orbifold for some $a\geq 1$. Then the local groups of $\Orb$ contained in strata of codimension $< 2^a$ have no $2$-torsion.
\end{lem}
\begin{proof} In any case $\Orb$ is simply connected and so we can realize it as an $\SOr(n)$ quotient of its frame bundle. Suppose a local group $\Gamma_p$ of $\Orb$ contains a subgroup $\Gamma_0$ of order $2$ and a fixed point subspace of codimension $<2^a$. The image of $\SOr(n)p \times E\SOr(n) \subset \mathrm{Fr}(\Orb) \times E\SOr(n)$ under the projection $\mathrm{Fr}(\Orb) \times E\SOr(n) \rightarrow B \Orb$ is a model for the classifying space of $\Gamma_p$. If we pull back the tangent bundle $T\Orb$ via the map $B\Gamma_0 \To B \Gamma_p \hookrightarrow B\Orb$ we get a bundle isomorphic to $V=E\Gamma_0 \times_{\Gamma_0} \R^n$. 

Suppose that the non-zero element $\iota \in \Gamma_0 \subset \SOr(n)$ has $2m<2^a$ times the eigenvalue $-1$. Then we can write $2m$ ``in binary'' as $2m=\sum_{j=1}^N 2^{e_j}$, where $0< e_1<e_2<\ldots e_N<a$. Furthermore, $V$ is a bundle over $\R\mathbb{P}^{\infty}\cong B\Gamma_0$ that decomposes as the sum of $2m$ canonical line bundles and $n-2m$ trivial line bundles. Thus the total Stiefel--Whitney class is given by
\[
(1+w)^{2m}=\prod_{j=1}^m (1+w)^{2^{e_j}}=\prod_{j=1}^m (1+w^{2^{e_j}})=1+w^{2^{e_1}}+R
\]
where $1$ is the generator of $H^0(\R\mathbb{P}^{\infty};\Z_2)$, $w$ is the generator of $H^1(\R\mathbb{P}^{\infty};\Z_2)\cong \Z_2$, and $R$ is a multiple of $w^{2^{e_1}+1}$. This implies that $w_{2^{e_1}}(V)\neq 0$. However, since $V$ is a pullback of $T\Orb$, and $\Orb$ is $2^{a-1}$-connected, we have $w_1(V)=\ldots =w_{2^{a-1}}(V)=0$. Moreover, since the Stiefel-Whitney classes $w_j(V)$, $j<2^a$ are generated by $w_1(V),\ldots, w_{2^{a-1}}(V)$ via Steenrod powers \cite[p.~94]{MS74}, it follows that $w_j(V)=0$ and thus for all $j<2^a$. This is a contradiction with $w_{2^{e_1}}(V)\neq 0$, since by assumption $2^{e_1}<2m<2^a$.
\end{proof}

It turns out that the absence of $2$-torsion has strong representation-theoretical implications as the following lemma shows. It for instance reflects the fact that any irreducible subgroup of $\SOr(3)$ contains elements of order $2$.

\begin{lem}\label{L:odd-order}
A real representation without trivial components of a finite group $G$ of odd order has even dimension.
\end{lem}
\begin{proof} Assume that the statement is false. Then there exists a nontrivial odd-dimensional irreducible real representation of $G$. Since it has odd dimension, its complexification $\rho: G\to U(V)$ (which we can turn to be unitary as $G$ is finite) is irreducible as well. Let $\chi_\rho:G\to \RR$ be the character of this representation: $\chi_\rho(g)=\operatorname{tr}(\rho(g))$. The fact that $G$ has odd order implies that the map $g \mapsto g^2$ is a bijection of $G$. This allows us to compute the Schur-indicator of $\rho$:
\[
		\iota_{\rho} := {1\over |G|}\sum_{g\in G}\chi_\rho(g^2)= {1\over |G|}\sum_{g\in G}\chi_\rho(g)=\langle \chi_\rho,\chi_{\rho_0}\rangle=0.
\]
Here $\rho_0$ denotes the trivial character and the last relation holds by the Schur orthogonality relation since $\rho$ nontrivial. Now the claim follows by contradiction since an irreducible complex representation with trivial Schur indicator cannot be realized over the real numbers \cite[Theorem~4.5.6]{Wol84}.
\end{proof}

The conclusion of Lemma \ref{L:odd-order} will be essential in our proof as it allows us to skip every second dimension. This will be necessary because later on we want to apply the following proposition to turn a highly connected orbifold into a $k$-connected orbifold without singular strata in codimension $\geq k+1$.

Recall that given an orbifold $\Orb$, a \emph{strong suborbifold} is a subset $\Sigma\subset \Orb$ such that for every $p\in \Sigma$ and every chart $\phi:U\to V\subset \Orb$ around $p$ (i.e. $U$ is a Riemannian manifold with an isometric action of a finite group $\Gamma$, $\phi$ is $\Gamma$-invariant and induces an isometry $U/\Gamma\simeq V$), the preimage $\phi^{-1}(\Sigma\cap V)$ is a submanifold fixed set-wise by the group $\Gamma$.

\begin{prp}\label{P:connected-complement}
Let $\Orb$ be a simply connected $n$-orbifold with a closed strong suborbifold ${\Sigma}\subset \Orb$ of codimension $k$. Then the pair $(\Orb, \Orb\setminus\Sigma)$ is $(k-1)$-connected in the orbifold topology.
\end{prp}
\begin{proof}
Let $D\Sigma$ denote a tubular neighbourhood of $\Sigma$, with smooth boundary $S\Sigma:=\partial D\Sigma$.
At any point $p\in \Sigma$, consider a local chart $\pi_p:\tilde{U}_p\to U_p$, where $\tilde{U}_p$ is an open set in $\RR^n$ on which the local group $\Gamma_p$ acts by fixing the (singleton) preimage $\tilde{p}$ of $p$. We can furthermore introduce an orbifold Riemannian metric such that the local group $\Gamma_p$ acts by isometries.

By definition of strong suborbifold, the preimage $\tilde{\Sigma}_p:=\pi_p^{-1}(\Sigma\cap U_p)$ is a submanifold of $\tilde{U}_p$, fixed (as a set) by the local group $\Gamma_p$. In particular, $\pi_p^{-1}(D\Sigma)=D(\tilde{\Sigma}_p)$ is preserved by the action of $\Gamma_p$, as well as $\partial D(\tilde{\Sigma}_p)=\pi_p^{-1}(S\Sigma)$.
\\

Consider now the frame bundle $\pi:\mathrm{Fr}(\Orb)\to \Orb$, and recall that on each local chart $U_p$, one has $\pi^{-1}(U_p)\simeq \tilde{U}_p\times_{\Gamma_p}\SOr(n)$. In particular we have:
\begin{align*}
\pi^{-1}(\Sigma\cap U_p)=& \tilde{\Sigma}_p\times_{\Gamma_p}\SOr(n)\\
\pi^{-1}(D\Sigma\cap U_p)=&D(\tilde{\Sigma}_p)\times_{\Gamma_p}\SOr(n)\\
\pi^{-1}(S\Sigma\cap U_p)=&(\partial D(\tilde{\Sigma}_p))\times_{\Gamma_p}\SOr(n)=\partial (D(\tilde{\Sigma}_p)\times_{\Gamma_p}\SOr(n))
\end{align*}
Since the projection $D(\tilde{\Sigma}_p)\to \tilde{\Sigma}_p$ is a $\Gamma_p$-equivariant trivial disc bundle, with sphere bundle $\partial D(\tilde{\Sigma}_p)\to \tilde{\Sigma}_p$, there is an induced map
\[
\pi^{-1}(D\Sigma\cap U_p)=D(\tilde{\Sigma}_p)\times_{\Gamma_p}\SOr(n)\to \tilde{\Sigma}_p\times_{\Gamma_p}\SOr(n)=\pi^{-1}(\Sigma\cap U_p)
\]
which is again a disc bundle, with sphere bundle given by $\pi^{-1}(S\Sigma\cap U_p)$.

In particular, $\pi^{-1}(D\Sigma)$ is a disk bundle over $\tilde{\Sigma}=\pi^{-1}(\Sigma)$, with sphere bundle and $\pi^{-1}(S\Sigma)$. Consider the map of fibrations:
\[
	\begin{xy}
		\xymatrix
		{
		 S^{k-1} \ar[d] \ar[r] & \mathbb{D}^k  \ar[d]\\
		\pi^{-1}(S\Sigma) \ar[r]\ar[d] & \pi^{-1}(D\Sigma) \ar[d]\\
		 \tilde{\Sigma} \ar@{=}[r]&\tilde{\Sigma}
		}
	\end{xy}
\]
since the inclusion $S^{k-1}\to \mathbb{D}^k$ is $(k-1)$-connected, it follows by the 5-lemma that the inclusion $\pi^{-1}(S\Sigma)\to \pi^{-1}(D\Sigma)$ is $(k-1)$-connected as well. Furthermore, since the inclusion $\pi^{-1}(S\Sigma)\to \pi^{-1}(D\Sigma)$ is $\SOr(n)$-equivariant, it induces fibrations
\[
	\begin{xy}
		\xymatrix
		{
		 &\pi^{-1}(S\Sigma) \ar[r]\ar[d] &\pi^{-1}(D\Sigma)\ar[d]&\\
		 \widehat{S\Sigma}\ar@{=}[r]& \pi^{-1}(S\Sigma)\times_{\SOr(n)}E\SOr(n)\ar[r]\ar[d]&\pi^{-1}(D\Sigma)\times_{\SOr(n)}E\SOr(n)\ar@{=}[r]\ar[d]& \widehat{D\Sigma}\\
		 &B\SOr(n)\ar@{=}[r]&B\SOr(n)&
		}
	\end{xy}
\]
where $\widehat{D\Sigma}, \widehat{S\Sigma}$ are the classifying spaces of $D\Sigma$ and $S\Sigma$, respectively. Applying the 5-lemma again to the long exact sequence in homotopy induced by the fibrations above, it follows that the map $\widehat{S\Sigma}\to \widehat{D\Sigma}$ induced by the inclusion ${S\Sigma}\subset{D\Sigma}$ is $(k-1)$-connected. In particular, $\pi_m^{\orb}(S\Sigma)\to \pi_m^\orb(D\Sigma)$ is an isomorphism for $m\leq k-2$, and surjective for $m=k-1$.
\\

Apply now Seifert--Van Kampen theorem to the partition $\mathcal{O}=U\cup V$ where $U= D\Sigma$ and $V=\mathcal{O}\setminus \Sigma$ (if $\Sigma$ is disconnected then we can apply Seifert--Van Kampen to each connected component of $\Sigma$ separately, or apply the groupoid version of Seifert--Van Kampen). Let us name the inclusions
\[
i_U:U\cap V\to U,\quad  i_V:U\cap V\to V,\quad j_U:U\to \Orb,\quad j_V:V\to \Orb.
\]
Then $\pi_*^\orb(U\cap V)\simeq \pi_*^\orb(S\Sigma)$, and $\pi_1^\orb(\Orb)$ is the pushout of
\[
\pi_1^\orb(U)\stackrel{{i_U}_*}{\longleftarrow} \pi_1^\orb(U\cap V)\stackrel{{i_V}_*}{\longrightarrow} \pi_1^\orb(V)
\]
but since the map ${i_U}_*$ is an isomorphism by the discussion above, it follows that the pushout of the diagram is isomorphic to $\pi_1^\orb(V)=\pi_1^\orb(\Orb\setminus \Sigma)$. Therefore, $\pi_1^\orb(\Orb\setminus \Sigma)\simeq \pi_1^\orb(\Orb)\simeq 0$.

Finally, by excision we have $H_{m}^\orb(\Orb, \Orb\setminus \Sigma)\simeq H_m^\orb(D\Sigma, S\Sigma)\simeq 0$ for $m\leq k-1$. By the relative Hurewicz isomorphism, this implies $\pi_m(\Orb, \Orb\setminus \Sigma)=0$ for all $m\leq k-1$, thus finishing the proof.
\end{proof}

Since the orbifold strata of $\Orb$ are special examples of strong suborbifolds, the previous proposition implies the following.

\begin{cor}\label{cor:complement_connected} Let $\Orb$ be a simply connected $n$-orbifold. Then the pair $(\Orb, \Orb\setminus\Sigma_{\geq k})$ is $(k-1)$-connected in the orbifold topology, where $\Sigma_{\geq k}$ denotes the union of all strata of codimension at least $k$.
\end{cor}

Now we can prove our first obstruction theorem.

\begin{proof}[Proof of Theorem \ref{thm:main_noncompact}]
In dimension $1$ no bad orbifolds exist. In dimension $2$ any bad orbifold is compact. Hence, we can assume that the dimension is at least $3$.

By Lemma \ref{L:improved-no-2-torsion} all local groups have odd order. Lemma \ref{L:odd-order} then implies that all singular strata have even codimension.

Assume by contradiction that $\mathcal{O}$ is not a manifold, and let $2k$ denote the lowest codimension of a singular stratum. Let $\Sigma_{\geq 2k+2}$ be the union of strata of codimension $\geq 2k+2$. By Corollary \ref{cor:complement_connected} and the fact that $\Orb$ is $2k$-connected, the complement $\Orb':=\Orb\setminus \Sigma_{\geq 2k+2}$ is $2k$-connected as well, and it contains only singular strata of codimension $2k$, which we call $\Sigma$.

By abuse of notation we denote by $\Sigma\subseteq \Orb'$ the connected component of $\Sigma$, and let $D\Sigma$ be a tubular neighbourhood of $\Sigma$ and $S\Sigma:=\partial D\Sigma$. By the proof of Proposition \ref{P:connected-complement}, the sets $D\Sigma$ and $\Sigma$ are in fact suborbifolds of $\Orb$, and the map $\widehat{S\Sigma}\to \widehat{D\Sigma}$ between their classifying spaces induced by the inclusion $S\Sigma\to D\Sigma$ is an $S^{2k-1}$-bundle. By the assumption on the structure of the singular set, $S\Sigma$ is a manifold and in fact $\widehat{S\Sigma}$($\simeq S\Sigma$) has finite dimensional cohomology. On the other hand, letting $\Gamma$ be the local group at $\Sigma$ one has that $H^*(\widehat{D\Sigma})=H^*_\orb(D\Sigma)$ is nonzero in infinitely many degrees.

From the Gysin sequence of $S\Sigma\to D\Sigma$ there exists an element $e\in H^{2k}_\orb(D\Sigma)$ with a long exact sequence
\[
\ldots H^{j}_\orb(S\Sigma)\to H^{j-(2k-1)}_\orb(D\Sigma)\stackrel{\cup e}{\longrightarrow}H^{j+1}_\orb(D\Sigma)\to H^{j+1}_\orb(S\Sigma)\to \ldots
\]
By choosing an index $j$ such that $H^{\geq j}_\orb(S\Sigma)=0$ and $H^{j-(2k-1)}_\orb(D\Sigma)\neq 0$, it follows that $e\neq 0$. By then choosing $j=2k-1$ one gets
\[
H^{2k-1}_\orb(S\Sigma)\to H^{0}_\orb(D\Sigma)\stackrel{\cup e}{\longrightarrow}H^{2k}_\orb(D\Sigma)\to H^{2k}_\orb(S\Sigma)\to \ldots
\]
i.e. $0 \neq e\in \ker(j_1^*:H^{2k}_\orb(D\Sigma)\to H^{2k}_\orb(S\Sigma))$. Apply now the Mayer-Vietoris sequence to the partition $(D\Sigma, \Orb'\setminus\Sigma)$ of $\Orb'$:
\[
0=H^{2k}_\orb(\Orb')\to H^{2k}_\orb(D\Sigma)\oplus H^{2k}_\orb(\Orb'\setminus\Sigma)\stackrel{j_1^*-j_2^*}{\longrightarrow} H^{2k}_\orb(S\Sigma)\to \ldots
\]
But this is a contradiction, because of $0\neq (e,0)\in \ker(j_1^*-j_2^*)$.
\end{proof}

In the compact case we can show more by exploiting Lefschetz duality.

\begin{prp}\label{prp:compact_even}
Any compact $(2n-2)$-connected $2n$-orbifold $\mathcal{O}$, $n\geq 3$, is actually a manifold.
\end{prp}
\begin{proof} By the preceding proof we can assume that every singular point is isolated. We denote the union of these singular points by $\Sigma$. Since $\Sigma$ is finite, for sufficiently small $\epsilon>0$ the $\epsilon$-balls around these singular points are disjoint. For such an $\epsilon$ we set $U=B_\epsilon(\Sigma)$.
By Proposition \ref{P:connected-complement} the complement $V=\Orb\setminus U$ is a compact $(2n-2)$-connected manifold with boundary, and by Lefschetz duality $H^i_{\orb}(V)=H_{2n-i}^{\orb}(V,V\cap U)$. In particular $H_i^\orb(V,V\cap U)=0$ for $i=2,\ldots,2n-1$. Via the long exact sequence in homology for the pair $(V,V\cap U)$ one gets for $i=2,\ldots,2n-1$ that
\[
0=H_i^\orb(V,V\cap U)\to H_{i-1}^{\orb}(V\cap U)\to H_{i-1}^{\orb}(V)=0.
\]
Hence, each component of $V\cap U$ is homotopy equivalent to a homology $(2n-1)$-sphere. On the other hand, $U\cap V\simeq S^{2n-1}/\Gamma_p$, where the local group $\Gamma_p$ acts freely on $S^{2n-1}$. Based on work by Zassenhaus \cite{Zas35} it was shown in \cite{Sje73} that there are no homology spheres in dimensions higher than $3$ which are covered by a genuine sphere. Because of $(2n-1)>3$ this implies that $\Gamma_p$ is trivial and so the claim follows.
\end{proof}

A slight modification of the proof of Proposition \ref{prp:compact_even} also shows the following statement.

\begin{prp}\label{prp:compact_odd}
Let $n>3$ be an integer that is not a power of $2$. Any compact $(2n-2)$-connected $(2n+1)$-orbifold $\mathcal{O}$ is actually a manifold.
\end{prp}
\begin{proof} By the proof of Theorem \ref{thm:main_noncompact} we can assume that the singular set is a disjoint union of embedded circles. By Proposition \ref{P:connected-complement} the complement $V=\Orb\setminus B_\epsilon(\Sigma)$ is a compact $(2n-2)$-connected manifold with boundary, and by Lefschetz duality $H^i_{\orb}(V)=H_{2n+1-i}^{\orb}(V,V\cap U)$. In particular $H_i^\orb(V,V\cap U)=0$ for $i=3,\ldots,2n$. Via the long exact sequence in homology for the pair $(V,V\cap U)$ one gets for $i=3,\ldots,2n-1$ that
\[
0=H_i^\orb(V,V\cap U)\to H_{i-1}^{\orb}(V\cap U)\to H_{i-1}^{\orb}(V)=0, 
\]
i.e. $H_{i}^{\orb}(V\cap U)=0$ for $i=2,\ldots,2n-2$. On the other hand, we know that each component of $V\cap U$ is homotopy equivalent to $S^1 \times S^{2n-1}/\Gamma_p$, where the local group $\Gamma_p$ acts freely on $S^{2n-1}$. The homological computation above together with the K\"unneth formula implies that $S^{2n-1}/\Gamma_p$ is a homology $(2n-1)$-sphere. Because of $(2n-1)>3$ this again implies that $\Gamma_p$ is trivial \cite{Sje73}.
\end{proof}


\begin{thebibliography}{[O'Ne66]}


\bibitem[ALR07]{AdemLeidaRuan} A. Adem, J. Leida\ and\ Y. Ruan, \emph{Orbifolds and stringy topology}, Cambridge Tracts in Mathematics, 171, Cambridge Univ. Press, Cambridge, 2007. 

\bibitem[ALR21]{ALR21} M. Amann, C. Lange\ and\ M. Radeschi, \emph{Odd-dimensional orbifolds with all geodesics closed are covered by manifolds}, Math. Ann. {\bf 380} (2021), no.~3-4, 1355--1386.

\bibitem[EMSS]{EMSS:77} J. Ewing, Suresh Moolgavkar, Larry Smith, R.E. Stong, \emph{Stable parallelizability of lens spaces}, Journal of Pure and Applied Algebra {\bf10} (1977), 177--191.


\bibitem[Fuk90]{MR1145256}  K. Fukaya, \emph{Hausdorff convergence of Riemannian manifolds and its applications}, Recent topics in differential and analytic geometry, 143--238, Adv. Stud. Pure Math., 18-I, Academic Press, Boston, MA, 1990.


\bibitem[Da11]{Da11}  M. W. Davis, Lectures on orbifolds and reflection groups, in {\it Transformation groups and moduli spaces of curves}, 63--93, Adv. Lect. Math. (ALM), 16, Int. Press, Somerville, MA, 2011.

\bibitem[Gh84]{Gh84}  É. Ghys, \emph{Feuilletages riemanniens sur les vari\'{e}t\'{e}s simplement connexes}, Ann. Inst. Fourier (Grenoble) {\bf 34} (1984), no.~4, 203--223.

\bibitem[GG88]{GG88} D. Gromoll\ and\ K. Grove, \emph{The low-dimensional metric foliations of Euclidean spheres}, J. Diff. Geom. {\bf 28} (1988), no.~1, 143--156.

\bibitem[Hae84]{Hae84}  A. Haefliger, \emph{Groupo\"\i des d'holonomie et classifiants}, Structure Transverse des Feuilletages, Ast\'erisque (1984), no. 116, 70--97.


\bibitem[La19]{La19} C. Lange, \emph{When is the underlying space of an orbifold a manifold?}, Trans. Amer. Math. Soc. {\bf 372} (2019), no.~4, 2799--2828.

\bibitem[La20]{La20} C.~Lange, \emph{Orbifolds from a metric viewpoint},  Geom. Dedicata, \textbf{209} (2020), 43--57.

\bibitem[Le15]{Le15} N. Lebedeva, \emph{Alexandrov spaces with maximal number of extremal points}, Geom. Topol. {\bf 19} (2015), no.~3, 1493--1521.

\bibitem[Ly13]{Ly13} A. Lytchak, \emph{On contractible orbifolds}, Proc. Amer. Math. Soc. {\bf 141} (2013), no.~9, 3303--3304.

\bibitem[LW16]{LW16} A. Lytchak\ and\ B. Wilking, \emph{Riemannian foliations of spheres}, Geom. Topol. {\bf 20} (2016), no.~3, 1257--1274.

\bibitem[MM91]{MM91} Y. Matsumoto\ and\ J. M. Montesinos-Amilibia, \emph{A proof of Thurston's uniformization theorem of geometric orbifolds}, Tokyo J. Math. {\bf 14} (1991), no.~1, 181--196.

\bibitem[Mil61]{Mil61} J. Milnor, A procedure for killing homotopy groups of differentiable manifolds, in {\it Proc. Sympos. Pure Math., Vol. III}, 39--55, (1961), American Mathematical Society, Providence, R.I.

\bibitem[MS74]{MS74} J. W. Milnor\ and\ J. D. Stasheff, {\it Characteristic classes}, Princeton University Press, Princeton, NJ, 1974. 

\bibitem[Qu71]{Qu71} D. Quillen, \emph{The spectrum of an equivariant cohomology ring. I}, Ann. of Math. (2) {\bf 94} (1971), 549--572; ibid. (2) {\bf 94} (1971), 573--602.


\bibitem[Sat56]{MR0079769} I. Satake, On a generalization of the notion of manifold, \emph{Proc. Nat. Acad. Sci. U.S.A.}, (1956) 42:359--363.
\bibitem[Sat57]{Satake} I. Satake, \emph{The Gauss-Bonnet theorem for $V$-manifolds}, J. Math. Soc. Japan {\bf 9} (1957), 464--492.

\bibitem[Sje73]{Sje73}  D. Sjerve, \emph{Homology spheres which are covered by spheres}, J. London Math. Soc. (2) {\bf 6} (1973), 333--336. 

\bibitem[Thu79]{Thurston} W. P. Thurston, \emph{The geometry and topology of three-manifolds}, lecture notes, Princeton University, 1979.

\bibitem[Wol84]{Wol84} J. A. Wolf, \emph{Spaces of constant curvature}, 5th ed., Publish or Perish, Inc., Houston, TX, 1984.

\bibitem[Wal62]{Wal62} C.T.C. Wall, \emph{Killing the middle homotopy groups of odd dimensional manifolds}, Trans. Amer. Math. Soc. {\bf 103} (1962), 421--433.

\bibitem[Ye15]{Ye15} D. Yeroshkin, \emph{Orbifold biquotients of $SU(3)$}, Differential Geom. Appl. {\bf 42} (2015), 54--76.

\bibitem[Zas35]{Zas35} H. Zassenhaus, \emph{\"Uber endliche Fastk\"orper}, Abh. Math. Sem. Univ. Hamburg {\bf 11} (1935), no.~1, 187--220.


\end{thebibliography}
\end{document}